\documentclass{amsart}
\usepackage{epsfig}
\usepackage{amsmath}

\newtheorem{theorem}{Theorem}[section]
\newtheorem{lemma}[theorem]{Lemma}

\newtheorem{corollary}[theorem]{Corollary}
\newtheorem{proposition}[theorem]{Proposition}

\theoremstyle{definition}
\newtheorem{definition}[theorem]{Definition}
\newtheorem{example}[theorem]{Example}

\newtheorem{note}[theorem]{Note}

\theoremstyle{remark}

\begin{document}

\title[Involutions and their progenies]
{Involutions and their progenies}

\author{Tewodros Amdeberhan}
\address{Department of Mathematics,
Tulane University, New Orleans, LA 70118}
\email{tamdeber@tulane.edu}

\author{Victor H. Moll}
\address{Department of Mathematics,
Tulane University, New Orleans, LA 70118}
\email{vhm@tulane.edu}


\date{\today}


\begin{abstract}
Any permutation has a disjoint cycle decomposition and concept generates an equivalence class
on the symmetry group called the cycle-type. The main focus of this work is on permutations of restricted cycle-types,
with particular emphasis on the special class of involutions and their partial sums. 
The paper provides generating functions,  determinantal expressions, asymptotic 
estimates as well as arithmetic and combinatorial properties.
\end{abstract}

\maketitle

\newcommand{\ba}{\begin{eqnarray}}
\newcommand{\ea}{\end{eqnarray}}
\newcommand{\ift}{\int_{0}^{\infty}}
\newcommand{\nn}{\nonumber}
\newcommand{\no}{\noindent}
\newcommand{\lf}{\left\lfloor}
\newcommand{\rf}{\right\rfloor}
\newcommand{\realpart}{\mathop{\rm Re}\nolimits}
\newcommand{\imagpart}{\mathop{\rm Im}\nolimits}
\newcommand{\ione}{I_{1}}
\newcommand{\itwo}{I_{2}}
\newcommand{\ch}{\qquad \mathbf{}}

\newcommand{\op}[1]{\ensuremath{\operatorname{#1}}}
\newcommand{\pFq}[5]{\ensuremath{{}_{#1}F_{#2} \left( \genfrac{}{}{0pt}{}{#3}
{#4} \bigg| {#5} \right)}}

\newtheorem{Definition}{\bf Definition}[section]
\newtheorem{Thm}[Definition]{\bf Theorem}
\newtheorem{Example}[Definition]{\bf Example}
\newtheorem{Lem}[Definition]{\bf Lemma}
\newtheorem{Cor}[Definition]{\bf Corollary}
\newtheorem{Prop}[Definition]{\bf Proposition}
\numberwithin{equation}{section}

\section{Introduction}
\label{sec-intro}
\setcounter{equation}{0}

For $n \in \mathbb{N}$, the group of permutations in $n$ symbols 
$\{ a_{1}, \, a_{2}, \cdots, a_{n} \}$ is called 
the symmetric group, denoted by   $\mathfrak{S}_{n}$.  A \textit{cycle} $\rho \in \mathfrak{S}_{n}$ is a permutation of the form, in a 
one-line notation, 
$\rho = ( a_{i_{1}} \, a_{i_{2}} \, \cdots a_{i_{r} })$. The notation indicates that all the entries of 
the cycle are distinct and 
$ \rho(a_{i_{j}}) = a_{i_{j+1}}$  for $1 \leq j \leq r-1$ and $\rho( a_{i_{r}}) = a_{i_{1}}$.  The 
cycle $\rho$ is said to have \textit{length} $r$, written as $r = L(\rho)$.  Every permutation $\pi \in \mathfrak{S}_{n}$ can be written as 
a product of cycles $\pi = \rho_{1}\rho_{2} \cdots  \rho_{m}$. This decomposition is not unique, but if the cycles are assumed 
to be disjoint and the lengths are taken in weakly decreasing order, then 
 $\{L(\rho_{1}), L(\rho_{2}), \cdots, L(\rho_{m})\}$ is uniquely determined by $\pi$, called the \textit{cycle type} of $\pi$.

The following notation is used: for $1 \leq \ell, \, t \leq n$, 
\begin{equation}
C_{n, \ell} = \{ \pi \in \mathfrak{S}_{n} \Big{|} \text{ with every cycle in } \pi \text{ of length at most } \ell \},
\label{cnl-def}
\end{equation}
\no
\begin{equation}
\alpha_{t}(\pi) = \text{ number of  cycles  in } \pi \in \mathfrak{S}_{n} \text{ of length }t.
\end{equation}
\no
and the cardinality of $C_{n, \ell}$ is denoted by $d_{n,\ell} = \# C_{n,\ell}$.

\begin{definition}
A permutation $\pi$ in $\mathfrak{S}_{n}$ is called an \textit{involution}
if $\pi^{2}(j) =  j$, for $1 \leq j \leq n$. The set of involutions in $\mathfrak{S}_{n}$ is 
denoted by $\text{Inv}(n)$. The cardinality of this set, denoted by $I_{1}(n)$, is 
called the \textit{involution number}.
\end{definition}

The factorization of $\pi$ as a product of disjoint cycles shows that any cycle in the factorization 
of an involution has length $1$ or $2$.  This implies $\text{Inv}(n) = C_{n,2}$ and thus 
$I_{1}(n) = d_{n,2}$. It follows that if $\pi \in \text{Inv}(n)$ is an 
involution, then $\alpha_{1}(\pi) + 2 \alpha_{2}(\pi) = n$. 

\begin{example}
Every permutation of $2$ symbols (a transposition)
is an involution and for $n=3$ there are $4$ involutions
\begin{equation}
\pi_{1} = (1)(2)(3), \, \pi_{2} = (12), \, \pi_{3} = (13), \pi_{4} = (23).
\end{equation}
\noindent
The cycles $(123)$ and $(132)$ are the only elements of $S_{3}$ that are not 
involutions. Therefore $I_{1}(2) = 2$ and $I_{1}(3) = 4$. 
\end{example}

Elementary properties of the numbers $I_{1}(n)$
are described in Section \ref{sec-ione}. These include a second order recurrence, an 
exponential generating function as well as an explicit finite sum. These 
 are generalized to the \textit{involution 
polynomials} $I_{1}(n;t)$ in Section \ref{sec-involution-poly} which are intimately linked
to the (probabilistic) Hermite polynomials defined by 
\begin{equation}
H_{n}(t) = n! \sum_{j=0}^{\lfloor n/2 \rfloor} \frac{(-1)^{j}}{j!(n-2j)!}
\frac{t^{n-2j}}{2^{j}}
\label{hermite-1a}
\end{equation}
\no
with generating function 
\begin{equation}
\sum_{n=0}^{\infty} H_{n}(t) \frac{x^{n}}{n!} = \exp\left(xt - \tfrac{1}{2}x^{2} \right).
\label{gen-hermite-0}
\end{equation}
\no
The involution polynomials have a combinatorial interpretation as the generating function for 
fixed points of permutation in $\mathfrak{S}_{n}$.   Arithmetic properties of $I_{1}(n)$ are presented 
in Section \ref{sec-arith-ione}. Particular emphasis is given to the $2$-adic valuation of $I_{1}(n)$. 
Recall that, for $x \in \mathbb{N}$ and $p$ prime, the \textit{$p$-adic valuation} of $x$, denoted by
$\nu_{p}(x)$, is the highest power of $p$ that divides $x$. 
An odd prime $p$ is called \textit{efficient}  if $p$ does not divide $I_{1}(j)$ for $0 \leq j \leq p-1$. 
Otherwise it is called \textit{inefficient}. The prime $p=3$ is efficient and  $p=5$ is inefficient 
since $I_{1}(4) = 10$.  A periodicity argument is used to show that $\nu_{p}(I_{1}(n) = 0$; i.e., 
$p$ is efficient. Morever, for a prime $p$, it is shown that either $p$ divides $I_{1}(n)$ infinitely 
often or never.  In the case of an inefficient prime, it is conjecture that the $p$-adic valuation of 
the sequence $I_{1}(n)$ can be given in terms of a tree $\mathbb{T}_{p}$. This phenomena is illustrated 
for the prime $p=5$.  It is an open question to characterize efficient (or 
inefficient) primes. The partial sums of $I_{1}(n)$, denoted by $a_{n}$,
are discussed in Section \ref{sec-sum-i1}. Their arithmetic properties 
are presented in Section \ref{sec-arith-a}.  For instance, an explicit 
expression for their $2$-adic valuation is given there. The valuations for odd 
primes are also conjectured to have a tree structure. This is illustrated in the case $p=5$.  
Section \ref{sec-permu} considers the statistics 
of the sequence $d_{n, \ell}$  in \eqref{cnl-def}.  This is a generalization
of $I_{1}(n) = C_{n,2}$. Finally, the asymptotic behavior of $d_{n, \ell} = |C_{n,\ell}|$ is given in Section \ref{sec-asymptotics}.

\section{Basic  properties of the involution numbers}
\label{sec-ione}
\setcounter{equation}{0}

This section discusses fundamental properties of $I_{1}(n)$. Some of them are well-known but proofs are 
included here for the convenience of the reader. 

\begin{theorem}
\label{thm-first-recu}
The sequence $\ione(n)$ satisfies the recurrence 
\begin{equation}
\ione(n) = \ione(n-1)+(n-1)\ione(n-2), \text{ for } n \geq 2,
\label{rec-invol}
\end{equation}
\no
with initial conditions $\ione(0) = \ione(1) = 1$. 
\end{theorem}
\begin{proof}
There are $\ione(n-1)$ involutions that fix $n$. The number of involutions 
that contain a cycle $(j \, n)$, with $1 \leq j \leq n-1$ is $n-1$ times 
the number of involutions containing the cycle $(n-1, \, n)$. This is 
$(n-1) \ione(n-2)$.
\end{proof}

The  recurrence above generates  the values 

\smallskip

\begin{center}
\begin{tabular}{||c|c|c|c|c|c|c|c|c|c|c|c||}
\hline 
$n$ & 0 & 1 & 2 & 3 & 4 & 5 & 6 & 7 & 8 & 9 & 10 \\
\hline 
$I_{1}(n)$ &1 & 1 & 2 & 4 & 10 & 26 & 76 & 232 & 764 & 2620 & 9496\\
\hline
\end{tabular}
\end{center}
\no 
This is sequence $A000085$ in OEIS.

\smallskip

The recurrence \eqref{rec-invol} now enables to write a generating function for  $\{ \ione(n) \}$.

\begin{theorem}
\label{one-bf}
The exponential generating function for $\ione(n)$ is 
\begin{equation}
\sum_{n=0}^{\infty} \frac{\ione(n)}{n!}x^{n} = \exp(x+  \tfrac{1}{2}x^{2}).
\label{genfun-ione}
\end{equation}
\end{theorem}
\begin{proof}
On the basis of \eqref{rec-invol} verify that both sides of \eqref{genfun-ione} satisfy
\newline $f'(x) = (1+x)f(x)$ and the value $f(0)=1$.
\end{proof}

Cauchy's product formula on $e^{x}$ and $e^{x^{2}/2}$ allows to express $I_{1}(n)$ 
as a finite sum.

\begin{corollary}
The involution numbers $I_{1}(n)$ are given by 
\begin{equation}
I_{1}(n) = \sum_{j=0}^{\lf n/2 \rf} \binom{n}{2j} \binom{2j}{j} \frac{j!}{ 2^{j}}.
\label{finite-sum1}
\end{equation}
\end{corollary}

The numbers $\begin{displaystyle} \binom{2j}{j} \frac{j!}{ 2^{j}} \end{displaystyle}$
appearing in \eqref{finite-sum1} are now shown to be of the same parity.

\begin{corollary}
For $j \in \mathbb{N}$, the numbers $(2j)!/(j!2^{j})$ are odd integers.
\end{corollary}
\begin{proof}
The identity
\begin{equation}
\frac{(2j)!}{j! 2^{j}} = \frac{(2j)(2j-1)\cdots (j+1)}{2^{j}}
\end{equation}
shows that the denominator is a power of $2$.  To compute this power, use 
 Legendre's formula
\begin{equation}
\nu_{2}(n!) = n - s_{2}(n),
\end{equation}
\no
where $s_{2}(n)$ is the sum of the digits of $n$ in its  binary  expansion. Therefore, 
\begin{equation}
\nu_{2} \left( \frac{(2j)!}{j!2^{j}} \right) =  \left( 2j-s_{2}(2j)  \right) - 
\left( j - s_{2}(j) \right) - j = 0,
\end{equation}
\no
in view of $s_{2}(2j) = s_{2}(j)$.
\end{proof}

A second recurrence for the involution numbers is presented next.

\begin{theorem}
\label{thm-second-recu}
For $n, m \in \mathbb{N}$, the involution numbers satisfy 
\begin{equation}
\ione(n+m) = \sum_{k \geq 0} k! \binom{n}{k} \binom{m}{k} \ione(n-k)\ione(m-k).
\label{recu-2a}
\end{equation}
\end{theorem}
\begin{proof}
Split up the set $[n+m]$ into two disjoint subsets $A$ and $B$, of 
$n$ and $m$ letters, respectively. Count the involutions in 
$\text{Inv}(n+m)$ acoording to the number $k$ of cross-permutations that 
make up a cycle $(ab)$, with $a \in A$ and $b \in B$. The letters 
$a$ and $b$ can be  chosen in $\binom{n}{k} \binom{m}{k}$ ways and 
$k!$ ways to place the $k$ cycles $(ab)$. The remaining elements in 
$A$ (respectively $B$) allow $\ione(n-k)$ (respectively $\ione(m-k)$)
involutions. To complete the argument, summing $k! \binom{n}{k} \binom{m}{k} \ione(n-k) \ione(m-k)$
over $k$.
\end{proof}

As a direct consequence of (in fact, equivalent to) Theorem \ref{thm-second-recu} the following 
analytic statement is recorded. This result bypasses the need for an otherwise messy chain 
rule for derivatives.

\begin{corollary}
Higher order derivatives of the function $f(x) = \exp\left( x + x^{2}/2 \right)$ are computed by the 
umbral
\begin{equation}
\frac{d^{m}}{dx^{m}} f(x) = f(x) \sum_{k=0}^{m} \binom{m}{k} I_{1}(m-k)x^{k}:= 
f(x) ( x+ I_{1})^{m}.
\end{equation}
\end{corollary}
\begin{proof}
From Theorem \ref{one-bf}, $\begin{displaystyle} \frac{d^{m}}{dx^{m}} f(x) = \sum_{n} I_{1}(n+m) 
\frac{x^{n}}{n!}. \end{displaystyle}$ The right-hand side of Theorem \ref{thm-second-recu} implies 
\begin{multline*}
\sum_{n} \sum_{k} k! \binom{n}{k} \binom{m}{k} I_{1}(n-k)I_{1}(m-k) \frac{x^{n}}{n!}  =   \\
\sum_{k} \binom{m}{k} I_{1}(m-k) x^{k} \sum_{n} I_{1}(n-k) \frac{x^{n-k}}{(n-k)!} \\
 =  f(x) \sum_{k} \binom{m}{k} I_{1}(m-k)x^{k}.
\end{multline*}
\no
The claim follows.
\end{proof}

The recurrence \eqref{recu-2a} is now used to prove periodicity of $I_{1}(n) \bmod p^{r}$.

\begin{theorem}
\label{thm-periodicity}
Let $p$ be a prime and $r \in \mathbb{N}$. Then $I_{1} \bmod p^{r}$ is a periodic 
sequence of period $p^{r}$.
\end{theorem}
\begin{proof}
Write $n = cp^{r}+t$ with $0 \leq t < p^{r}$. Theorem \ref{thm-second-recu} gives
\begin{equation}
I_{1}(cp^{r}+t) = \sum_{k=0}^{t} k! \binom{cp^{r}}{k} \binom{t}{k} I_{1}(cp^{r}-k) I_{1}(t-k).
\end{equation}
\no
For $k>0, \,\binom{cp^{r}}{k} k! = (cp^{r})(cp^{r}-1) \cdots (cp^{r}-k+1) \equiv 0 \bmod p^{r}$
yields
\begin{equation}
I_{1}(cp^{r}+t) \equiv I_{1}(cp^{r}) I_{1}(t) \bmod p^{r}.
\label{red-1}
\end{equation}
\no
Using Theorem \ref{thm-second-recu} again 
\begin{equation}
I_{1}(2p^{r}) = \sum_{k=0}^{p^{r}} k! \binom{p^{r}}{k}^{2} I_{1})(p^{r}-k)^{2} 
\equiv I_{1}(p^{r})^{2} \bmod p^{r}
\end{equation}
\no
and then induction on $c$  gives
\begin{equation}
I_{1}(cp^{r}) \equiv I_{1}(p^{r})^{c} \bmod p^{r}.
\label{red-2}
\end{equation}

The next step is to show that  $I_{1}(p^{r}) \equiv 1 \bmod p^{r}$.  Then \eqref{red-1} and 
\eqref{red-2} imply the required periodicity.  Observe first that 
for $m \not \equiv 0 \bmod p$, 
\begin{equation}
\binom{p^{r}}{m} = \frac{p^{r}}{m} \binom{p^{r}-1}{m-1} \equiv 0 \bmod p^{r},
\end{equation}
\no
so that 
\begin{equation}
I_{1}(p^{r}) \equiv \sum_{m=0}^{r-1} \binom{p^{r}}{2mp} \binom{2mp}{mp} 
\frac{(mp)!}{2^{mp}} \bmod p^{r},
\end{equation}
\no
where the upper bound arises  from $\nu_{p}((mp)!) \geq m + \lf m/r \rf \geq r$ if 
$m \geq r$. 

The final step is to show that 
\begin{equation}
\label{last-step}
\nu_{p} \binom{p^{r}}{2mp} = \begin{cases}
    r-1 &  \text{ if } m \not \equiv 0 \bmod p \\
    r-2 & \text{ if } m \equiv 0 \bmod p.
      \end{cases}
 \end{equation}
 \no
 This would imply $I_{1}(p^{r}) \equiv 1 \bmod p^{r}$ since $\nu_{p}((mp)!) \geq 2$ for 
 $m \geq 2$. The periodicity of $I_{1}(n) \bmod p^{r}$ follows from here. 
 
 \smallskip
 
 To prove \eqref{last-step}, recall Legendre's formula 
 \begin{equation} 
 \nu_{p}(x!) = \frac{x - s_{p}(x)}{p-1}
 \end{equation}
 \no
  where  $s_{p}(x)$ is the digit sum of $x$ in base $p$. This gives 
 \begin{eqnarray}
 \nu_{p} \left( \binom{p^{r}}{2mp} \right) & = & 
 \frac{-s_{p}(p^r) + s_{p}(2mp) + s_{p}(p^{r}-2mp)}{p-1} \label{val-bin} \\
 & = & \frac{-1+ s_{p}(2m) + s_{p}(p^{r-1}-2m)}{p-1}. \nonumber 
 \end{eqnarray}
 \no
 Write  $\begin{displaystyle} 2m = \sum_{i=0}^{r-2} u_{i}p^{i}
 \end{displaystyle}$ with  $0 \leq u_{i} \leq p-1$. Then 
 \begin{eqnarray}
 p^{r-1}-2m & = & p^{r-1} - \sum_{i=0}^{r-2} u_{i}p^{i}
   = 1 + \sum_{i=0}^{r-2} (p-1-u_{i})p^{i}  \nonumber \\
   & = & (p-u_{0}) + \sum_{i=1}^{r-2} (p-1-u_{i})p^{i}.  \nonumber 
\end{eqnarray}
\no
If $m \not \equiv 0 \bmod p$, then 
\begin{equation}
s_{p}(p^{r-1}-2m) = (p-u_{0})+ \sum_{i=1}^{r-2} (p-1-u_{i}).
\end{equation}
\no
On the other hand, if $\nu_{p}(m) = a$, a direct calculation leads to
\begin{equation}
s_{p}(p^{r-1}-2m) = (p-u_{a})+ \sum_{i=a+1}^{r-2} (p-1-u_{i}).
\end{equation}
\no
The claim \eqref{last-step} now follows from \eqref{val-bin}.
\end{proof}

\begin{corollary}
\label{p-doesnot}
 Assume $\ione(n) \not \equiv 0 \bmod p$ for $0 \leq n \leq p-1$ and 
 $p$ an odd prime. Then $\nu_{p}(\ione(n)) \equiv 0$.
\end{corollary}

\begin{corollary}
A prime $p$ divides the sequence $\ione(n)$ infinitely often or never at all.
\end{corollary}

In the process of discovering the previous congruences, the following result was
obtained by the authors. Even though it is not related yet to the material that follows, it 
is of intrinsic interest and thus placed here for future use. In the sequel, $\lambda\vdash n$ means
$\lambda$ is a partition of $n$.

\begin{proposition}
Let $\lambda = (\lambda_{1}, \cdots, \lambda_{k}) \vdash n$ with $\lambda_{1} \geq \cdots \geq \lambda_{k} \geq 1$. 
Denote $ \binom{p n}{p \lambda} = \binom{pn}{p \lambda_{1}, \,  \cdots, \,  p \lambda_{k}}$. 

\noindent
a) If $p \geq 3$ is a prime, then $\binom{pn}{p \lambda} \equiv \binom{n}{\lambda} \bmod p^{2}$. 

\smallskip

\noindent
b) If $p \geq 5$ is a prime, then $\binom{pn}{p \lambda} \equiv \binom{n}{\lambda} \bmod p^{3}$. 

\end{proposition}
\begin{proof}
The case $k=2$ is considered first. Take an $n \times p$ rectangular grid. Choose $pb$ of these squares and paint  them
red. One option is to paint $b$ entire rows red, call this type-1. This can be done in $\binom{n}{b}$ different ways. In all other 
cases, there exist at least two rows each consisting of $t$ red squares, where $0 < t < p$.  Two such coloring are considered 
equivalent if one is produced from the other by a cyclic  shift of the squares in
each  row independently. This generates equivalence classes and the number of elements in each class is then divisible by 
$p^{2}$. Thus, modulo $p^{2}$, one only type-1 coverings remain.

To prove the general case, choose $p \lambda_{i}$ of these squares and paint them with color $c_{i}, \, 1 \leq i \leq k$, in 
$\binom{pn}{p \lambda}$ ways. Then proceed as in the case $k=2$. The general case also follows from the special case
$k=2$ and  the identity 
\begin{eqnarray*}
\binom{pn}{p \lambda} & = &  \binom{pn}{p \lambda_{1}} \binom{p(n - \lambda_{1})}{p \lambda_{2}} \cdots 
\binom{p(n- \lambda_{1} - \cdots - \lambda_{k-1}}{p \lambda_{k}} \label{congr11} \\
& \equiv & \binom{n}{\lambda_{1}} \binom{n - \lambda_{1}}{\lambda_{2}} \cdots \binom{n- \lambda_{1} - \cdots - \lambda_{k-1}}
{\lambda_{k}} \bmod p^{2} \nonumber \\
& = & \binom{n}{\lambda}.
\end{eqnarray*}

\medskip 

A similar argument reveals the second congruence.
\end{proof}

\section{The involution polynomials}
\label{sec-involution-poly}
\setcounter{equation}{0}

This section introduces a sequence of polynomials generalizing the involution numbers $I_{1}(n)$.
To this end, modify \eqref{rec-invol} so that $I_{1}(n;1) = I_{1}(n)$. 

\begin{definition}
The \textit{involution polynomials} $\ione(n;t)$ are defined by the 
recurrence 
\begin{equation}
\ione(n;t) = t \ione(n-1;t) + (n-1) \ione(n-2;t),
\label{invo-recu0}
\end{equation}
\noindent
with initial conditions $\ione(0;t) = 1$ and $\ione(1;t) = t$.
\end{definition}

\begin{proposition}
The involution polynomials are expressible as
\begin{equation}
\ione(n;t) = \sum_{j=0}^{\lf \frac{n}{2} \rf} \binom{n}{2j} \frac{(2j)!}{2^{j} j!} t^{n-2j}.
\label{form-invo}
\end{equation}
\end{proposition}
\begin{proof}
A direct calculation shows that the right-hand side of \eqref{form-invo} satisfies the 
recurrence \eqref{invo-recu0} with the same initial conditions as $I_{1}(n,t)$.
\end{proof}

\begin{Thm}
There is an  exponential generating function for the involution polynomials 
\begin{equation}
\sum_{n=0}^{\infty} I_{1}(n;t) \frac{x^{n}}{n!} = \text{exp}\left( tx + \tfrac{1}{2}x^{2} \right).
\end{equation}
\end{Thm}
\begin{proof}
Multiply the recurrence \eqref{invo-recu0} by $x^{n}/n!$ and sum over $n \geq 2$ to produce 
\begin{equation}
\sum_{n=2}^{\infty} I_{1}(n;t) \frac{x^{n}}{n!}  = \sum_{n=1}^{\infty} tI_{1}(n;t) \frac{x^{n+1}}{(n+1)!} + 
\sum_{n=0}^{\infty} I_{1}(n;t) \frac{x^{n+2}}{(n+1)n!}.
\end{equation}
\no
Denote the generating function by $h(x,t)$. The  recurrence implies
\newline  $\begin{displaystyle} \frac{\partial h}{\partial x} = (x+t)h \end{displaystyle}$ and the proof follows from
a standard argument.
\end{proof}

\begin{note}
The generating function \eqref{gen-hermite-0} shows the relation
\begin{equation}
I_{1}(n;t) = \imath^{n} H_{n}(- \imath t)
\end{equation}
\no
between the involution polynomials $I_{1}(n;t)$ and the Hermite polynomials $H_{n}(t)$.
\end{note}

The next result offers a  combinatorial interpretations of the involution
polynomials.

\begin{proposition}
The involution polynomials can be expressed as
\begin{equation}
\ione(n;t) = \sum_{\pi \in \text{Inv}(n)} t^{\alpha_{1}(\pi)},
\label{form-i1}
\end{equation}
\noindent
where $\alpha_{1}(\pi)$ is the number of fixed points of $\pi$.
\end{proposition}
\begin{proof}
Let $g_{n}(t)$ be the right-hand side in \eqref{form-i1}.  Rearrange the set of involutions
$\pi \in \text{Inv}(n)$ into two groups according to whether $\pi(n) = n$ or not. In the 
first case $\pi = \pi_{1}$ with $\pi_{1} \in \text{Inv}(n-1)$. The involution 
$\pi$ has the same number of 
$2$-cycles as $\pi_{1}$ and the extra fixed point $n$. Therefore the term $t^{c_{1}(\pi)}$ 
in $g_{n}(t)$ cancels a unique term in $tg_{n-1}(t)$. In the second case, let $\pi(n) = k$ 
with $1 \leq k \leq n-1$. Then $\pi$ is $\pi_{2}$ times the cycle $(nk)$; that 
is, $\pi = \pi_{2}(nk)$, with $\pi_{2} \in \text{Inv}(n-2)$. The permutation 
$\pi_{2}$ has the same number of fixed points as $\pi$. Thus, $t^{c_{1}(\pi)}$ in $g_{n}(t)$ cancels a 
unique term in $g_{n-2}(t)$.  Summing over $n$ gives the relation 
\begin{equation}
g_{n}(t) = tg_{n-1}(t) + g_{n-2}(t),
\end{equation}
\no
since every term on both sides has been canceled in the previous description.  The polynomials 
$g_{n}(t)$ and  $I_{1}(n;t)$ satisfy the same recurrence with matching initial conditions. This establishes
the assertion. 
\end{proof}

\section{Arithmetic properties of the  numbers $\ione(n)$. }
\label{sec-arith-ione}
\setcounter{equation}{0}

This section discusses the $p$-adic valuation of the sequence $\{ \ione(n) \}$.  
The analysis begins with the prime $p=2$.  

\begin{theorem}
\label{valuation-2}
The $2$-adic valuation of $\ione(n)$ is given by 
\begin{equation}
\nu_{2}(\ione(n)) = \begin{cases}
k & \text{ if } n = 4k \\
k & \text{ if } n = 4k+1 \\
k+1 & \text{ if } n = 4k+2 \\
k+2 & \text{ if } n = 4k+3 
\end{cases}
\end{equation}
\no
This is equivalent to 
$\begin{displaystyle}
\nu_{2}(\ione(n)) = \lf \frac{n}{2} \rf - 2 \lf \frac{n}{4} \rf + 
\lf \frac{n+1}{4} \rf.
\end{displaystyle}$
\end{theorem}
\begin{proof}
Let $n \in \mathbb{N}$ and assume the result is valid up to $n-1$. The proof 
is divided into four cases according to the residue of $n$ modulo $4$.  The 
symbol $O_{i}$ stands for an odd number.

\smallskip

\noindent
\textbf{Case 1}: $n = 4k$. The induction hypothesis states that 
\begin{equation}
\nu_{2}(\ione(n-1)) = k+1, \, \nu_{2}(\ione(n-2)) = k \text{ and } 
\nu_{2}(n-1) = 0.
\end{equation}
\noindent
The recurrence \eqref{rec-invol} implies  that 
$\ione(n) = 2^{k+1}O_{1} + 2^{k}O_{2} = 2^{k} \left( 2O_{1} + O_{2} \right),$
for some $O_{1}, \, O_{2}$ odd integers. This proves $\nu_{2}(\ione(n)) = k$.

\smallskip
 
\noindent
\textbf{Case 2}: $n =4k+1$.  The argument is similar to Case 1.

\smallskip 

\noindent
\textbf{Case 3}: $n =4k+2$. By  induction hypothesis, $\ione(n-1) = 2^{k}O_{1}$ and $\ione(n-2) = 2^{k} O_{2}$. The 
 recurrence \eqref{rec-invol} now yields 
$\ione(n) = 2^{k} \left( O_{1} + O_{2}O_{3} \right)$, with 
$O_{1} + O_{2}O_{3}$ even, so that $\nu_{2}(I_{1}(n)$ is not determined from
here. It is necessary to iterate \eqref{rec-invol} to obtain $\ione(n) = n \ione(n-2) + (n-2) \ione(n-3)$.
The result now follows immediately.

\smallskip

\noindent
\textbf{Case 4}: $n = 4k+3$. The recurrence \eqref{rec-invol} now needs to be iterated twice to 
produce $\ione(n) = 2(n-1) \ione(n-3) + n(n-3) \ione(n-4)$. Induction 
gives $\ione(n) = 2^{k+2} \left[ O_{1}+ 2^{1+ \nu_{2}(k)}O_{2} 
\right]$, showing that $\nu_{2}(\ione(n)) = k+2$. 

\medskip

An alternative proof follows from the recurrence \eqref{recu-2a}. Write 
$n = 4k+r$ for $0 \leq r \leq 3$ and proceed by induction on $k$. The 
result follows directly from the identities 
\begin{eqnarray*}
\ione(4k+1) & = & \ione(4k) + 4 k \ione(4k-1) \\
\ione(4k+2) & = & 2 \ione(4k) + 8k \ione(4k-1) + 4k(4k-1) \ione(4k-2) \\
\ione(4k+3) & = & 4 \ione(4k) + 24 k \ione(4k-1) + 
12 k(4k-1) \ione(4k-2) \\ 
& & \quad \quad + 6 \binom{4k}{3} \ione(4k-3) \\
\ione(4k+4) & = & 10 \ione(4k) + 64k \ione(4k-1) + 48k (4k-1) \ione(4k-2) 
\\ 
 & &   \quad \quad + 24 \binom{4k}{3} \ione(4k-3) + 24 \binom{4k}{4} \ione(4k-4).
\end{eqnarray*}
\end{proof}

\medskip

The case of $\nu_{p}(\ione(n))$ for $p$ an odd prime is considered next.  Lemma 
\ref{p-doesnot} shows that if $\ione(n) \not \equiv 0 \bmod p$ for 
$0 \leq n \leq p-1$, then $\nu_{p}(\ione(n)) \equiv 0$. 

\begin{definition}
The prime $p$ is called \textit{efficient} if $I_{1}(n) \not \equiv 0 \bmod p$, for 
every $n$ in the range $0 \leq n \leq p-1$. Otherwise, it is called \textit{inefficient}.
\end{definition}

Lemma \ref{p-doesnot} shows that $\nu_{p}(I_{1}(n)) \equiv 0$ if $p$ is an 
efficient prime.

\begin{example}
The values $\ione(0)=1, \, \ione(1) = 1, \, \ione(2) = 2$ show that $p=3$ is efficient. Therefore 
$\nu_{3}( \ione(n)) \equiv 0$. The prime $p=5$ is inefficient since
$\ione(4) = 10$ is divisible by $5$. Similarly $p=7$ is efficient, in view of the table 
\begin{equation*}
\begin{tabular}{|c|ccccccc|}
\hline
$n$ &  0 & 1 & 2 & 3 & 4 & 5 & 6  \\
\hline 
$\ione(n)$ & 1 & 1 & 2 & 4 & 10 & 26 & 76 \\
\hline 
$\text{Mod}(\ione(n),7)$ & 1 & 1 & 2 & 4 & 3 & 5 & 6 \\
\hline 
\end{tabular}
\end{equation*}
\end{example}

Among  the first $100$ primes, there are $62$ inefficient ones. These 
are listed in the table below.

\begin{equation}
\begin{tabular}{|cccccccc|}
\hline
5 & 13 & 19 & 23 & 29 & 31 & 43 & 53  \\
\hline 
59 & 61 & 67 & 73 & 79 & 83 & 89 & 97  \\
\hline 
103 & 131 & 137 & 151 & 157 & 163 & 173 & 179 \\
\hline 
181 & 191 & 197 & 199 & 211 & 229 & 233 & 239   \\
\hline 
241 & 281 & 293 & 307 & 317 & 347 & 359 & 367   \\
\hline 
373 & 379 & 389 & 397 & 409 & 419 & 421 & 431   \\
\hline 
433 & 443 & 449 & 457 & 461 & 463 & 479 & 487  \\
\hline 
491 & 499 & 509 & 521 & 523 & 541 &  &    \\
\hline
\end{tabular}
\end{equation}

\medskip

The $p$-adic valuation $\nu_{p}(I_{1}(n))$ for inefficient primes is (conjecturally) described by a tree structure $\mathbb{T}_{p}$ and certain modular classes.   The case $p=5$ is prototypical. 

Each vertex $V$ of the 
tree $\mathbb{T}_{5}$ corresponds to a subset of $\mathbb{N}$.  The vertex $V$ is called \textit{terminal} if 
$\{ \nu_{5}(I_{1}(n)): \, n \in V \}$ reduces to a single value; that is, $\nu_{5}(I_{1}(n))$ is independent of $n \in V$; otherwise it 
is called \textit{non-terminal}. The description of the tree $\mathbb{T}_{5}$ uses the notation $\Omega_{5}:= \{ 0, \, 1, \, 2, \, 3, \, 4 \}$. 

The construction 
begins with a \textit{root vertex} $V_{0}$ that  represents all $\mathbb{N}$.  Since 
$\nu_{5}(I_{1}(n))$ is not a 
constant function, the vertex $V_{0}$ is non-terminal. 
The root is now split into five different vertices, denoted by $V_{1,k}: \, 
k \in \Omega_{5}$, with 
\begin{equation}
V_{1,k} = \{ n \in \mathbb{N}: \, n \equiv k \bmod 5 \}.
\end{equation}
\no
These five vertices form the \textit{first level}.  Theorem \ref{thm-periodicity} shows that 
\begin{equation}
I_{1}(k+5n) \equiv I_{1}(k) \bmod 5
\end{equation}
\no
for $k \in \Omega_{5}$. The values $I_{1}(0) = 1, \, I_{1}(1) = 1, \, I_{1}(2) = 2, \, I_{1}(3) = 4, \, I_{1}(4) = 10$
give 
\begin{equation}
\nu_{5}(V_{1,k}) = 0 \text{ for } 0 \leq k \leq 3 \text{ and }\nu_{5}(V_{1,4}) \geq 1.
\end{equation}
\no
Thus, $V_{1,k}$ is a terminal vertex for $0 \leq k \leq 3$ and $V_{1,4}$ is non-terminal.  

In order to determine the valuation of numbers associated to the vertex $V_{1,4}$, that is, numbers of the 
form $5n_{1}+4$, split the index $n_{1}$ according to its residue modulo $5$ and write $5n_{1}+4 = 
5^{2}n_{2}+5k+4$, with $k \in \Omega_{5}$.  Then 
\begin{equation}
\nu_{5}(I_{1}(5^{2}n_{2}+5k+ 4 )) \geq 1, \text{ for } k \in \Omega_{5}.
\end{equation}
\no
The \textit{second level} is formed by vertices $V_{2,k}$ corresponding to the sets 
\newline $\begin{displaystyle}  \{ n \in \mathbb{N}: \, n \equiv 5k+4 \bmod 5^{2} \}. \end{displaystyle}$
Theorem \ref{thm-periodicity} gives 
\begin{equation}
I_{1}(5^{2}n_{2} + 5k+4 ) \equiv I_{1}(5k+4) \bmod 5^{2}, \text{ for every } k \in \Omega_{5}.
\end{equation}
\no
Therefore if  $I_{1}(5k+4) \not \equiv 0 \bmod 5^{2}$, it follows that $\nu_{5}(V_{2,k}) = 1$ and $V_{2,k}$ is a
terminal vertex. The values 
\begin{equation}
I_{1}(4) \equiv 10, \, I_{1}(9) \equiv 20 , \, I_{1}(14) \equiv 5, \, I_{1}(19) \equiv 15,\, I_{1}(24) \equiv 0 \bmod 5^{2},
\end{equation}
\no
show that $\nu_{5}(V_{2,k}) = 1, \text{ for } k \in \Omega_{5}, \, k \neq 4$ and, in the single remaining case,
$\nu_{5}(V_{2,4}) \geq 2$. 

\medskip

\no
\textbf{Conjecture}.  Assume $p$ is an inefficient prime. Then, for every 
$n \in \mathbb{N}$, the $n$-th level of the tree $\mathbb{T}_{p}$ contains a single non-terminal vertex. This 
level contains $p-1$ vertices with valuation $n-1$ and the single non-terminal vertex has valuation at least $n$.
This determines the tree $\mathbb{T}_{p}$ and the valuations $\nu_{p}(I_{1}(n))$.

\section{Partial sums of involution numbers}
\label{sec-sum-i1}
\setcounter{equation}{0}

What happens if the term $\binom{n}{2k}$ is replaced by $\binom{n}{2k+1}$ in the formula 
\begin{equation}
I_{1}(n) = \sum_{k \geq 0} \frac{(2k)!}{k! \, 2^{k}} \binom{n}{2k} ?
\end{equation}
\noindent
It is  perhaps convenient to also shift $n$ and define
\begin{equation}
a_{n} = \sum_{k \geq 0} \frac{(2k)!}{k! \, 2^{k}} \binom{n+1}{2k+1}.
\end{equation}
\noindent
The next result shows that $a_{n}$ is actually closely tied to $I_{1}$. 

\begin{theorem}
If $n \in \mathbb{N}$, then
\begin{equation}
a_{n} = \sum_{j=0}^{n} I_{1}(j).
\end{equation}
\end{theorem}
\begin{proof}
This is immediate from a simple binomial identity so that 
\begin{eqnarray*}
\sum_{j=0}^{n} I_{1}(j)  & = & \sum_{j=0}^{n} \sum_{k=0}^{\lfloor j/2 \rfloor} \binom{j}{2k} 
\frac{(2k)!}{k! 2^{k}} \\
& = & \sum_{k=0}^{\lfloor n/2 \rfloor} \frac{(2k)!}{k! 2^{k}} \sum_{j = \lfloor k/2 \rfloor}^{n} 
\binom{j}{2k} \\
& = & \sum_{k=0}^{\lfloor n/2 \rfloor} \frac{(2k)!}{k! 2^{k}} \binom{j+1}{2k+1} \\
& = & a_{n}.
\end{eqnarray*}
\end{proof}

A recurrence for $a_{n}$ is routinely generated by  the WZ-method (see 
\cite{nemesi-1997a,petkovsek-1996a}). 

\begin{Prop}
The sequence $a_{n}$ satisfies the recurrence 
\begin{equation}
a_{n} = 2a_{n-1} + (n-2)a_{n-2} -(n-1)a_{n-3}, \text{ for } n \geq 3,
\label{recur-a}
\end{equation}
\noindent
with initial conditions $a_{0}=1, \, a_{1}=2$ and $a_{2}=4$.
\end{Prop}

The first few values are tabulated below.

\smallskip

\begin{center}
\begin{tabular}{||c|c|c|c|c|c|c|c|c|c|c|c||}
\hline 
$n$ & 0 & 1 & 2 & 3 & 4 & 5 & 6 & 7 & 8 & 9 & 10 \\
\hline 
$a_{n}$ &1 & 2 & 4 & 8 & 18 & 44& 120 & 352 & 1116 & 3736 & 13232 \\
\hline
\end{tabular}
\end{center}
\no 
This sequence does not appear  in OEIS. 

\smallskip

Given any sequence $\{ q_{n} \}$ with ordinary generating function $f(x)$, then the partial sums 
$q_{1}+\cdots + q_{n}$ have the ordinary generating function $f(x)/(1-x)$. 
The corresponding statement for exponential generating functions is 
given below. 

\begin{Lem}
\label{exp-gen1}
If $\begin{displaystyle}w(x) = \sum_{n \geq 0} c_{n} \frac{x^{n}}{n!}
\end{displaystyle}$ and 
$\begin{displaystyle}u_{n} = \sum_{k=0}^{n} c_{k}\end{displaystyle}$, 
then 
\begin{equation}
 \sum_{n=0}^{\infty} u_{n} \frac{x^{n}}{n!}  = 
w(x) + e^{x} \int_{0}^{x} e^{-t} w(t) \, dt.
\end{equation}
\end{Lem}
\begin{proof}
Start with $u_{n} = c_{n} + u_{n-1}$, multiply through by $x^{n-1}/(n-1)!$ and 
sum over $n$. The outcome is the differential equation $g'(x) - g(x) = w(x)$.
Now solve this linear differential equation to obtain the result.
\end{proof}

\begin{Cor}
\label{cor-expgf}
The exponential generating function for the sequence $\{ a_{n} \}$ is 
\begin{equation}
\sum_{n=0}^{\infty} a_{n} \frac{x^{n}}{n!} = 
e^{x + x^{2}/2}  + e^{x} \int_{0}^{x} e^{t^{2}/2} \, dt. \ch
\end{equation}
\end{Cor}
\begin{proof}
Using $\begin{displaystyle} \sum_{n=0}^{\infty} I_{1}(n) \frac{x^{n}}{n!} = \exp \left( x + x^{2}/2 \right)
\end{displaystyle}$, the claim follows from Lemma \ref{exp-gen1}.
\end{proof}

\begin{Cor}
\label{cor-cauchy1}
The sequence $\{ a_{n} \}$ satisfies 
\begin{equation}
\sum_{k=1}^{n} (-1)^{n-k} \binom{n}{k} a_{k-1} = 
\begin{cases}
(2m)!/2^{m} \, m! & \quad \text{ if } n = 2m+1, \\
0, & \quad \text{ if } n = 2m.
\end{cases}
\end{equation}
\end{Cor}
\begin{proof}
Using the notation of Lemma  \ref{exp-gen1},
\begin{equation*}
\int_{0}^{x} e^{t^{2}/2} \, dt = e^{-x} [ g(x) - w(x) ] 
= e^{-x} \sum_{n=0}^{\infty} [a_{n} - I_{1}(n)] \frac{x^{n}}{n!} =
e^{-x} \sum_{n=1}^{\infty} a_{n-1} \frac{x^{n}}{n!}.
\end{equation*}
\no
Now write $e^{-x}$ as a series and multiply out to arrive at the assertion.
\end{proof}

\begin{Cor}
\label{cor-cauchy2}
The following identity holds:
\begin{equation}
\sum_{j=1}^{m} \binom{2m}{2j} a_{2j-1} = \sum_{j=1}^{m} \binom{2m}{2j-1}a_{2j-2}.
\end{equation}
\end{Cor}
\begin{proof}
This is Corollary \ref{cor-cauchy1} for $n=2m$.
\end{proof}

\section{Arithmetic properties of the sequence $a_{n}$}
\label{sec-arith-a}
\setcounter{equation}{0}

The next statement  is the corresponding counterpart to Theorem \ref{valuation-2}.

\begin{theorem}
\label{valuation-an}
The $2$-adic valuation of the sequence $a_{n}$ is given by 
\begin{equation}
\nu_{2}(a_{n}) = \begin{cases}
k, & \quad \text{ if } n = 4k-3, \\
k+1, & \quad \text{ if } n = 4k-2, \\
k, & \quad \text{ if } n = 4k, \\
\nu_{2}(k) + k + 2, & \quad \text{ if } n = 4k-1.
\end{cases}
\end{equation}
\end{theorem}
\begin{proof}
The inductive proof distinguishes the four values of $n$ modulo $4$. 

\smallskip

\noindent
\textbf{Case 1}: $n = 4k-3$. Then, the induction hypothesis shows that
\begin{equation}
a_{n-1} = 2^{k-1}O_{1}, \, a_{n-2} = 2^{k+1+\nu_{2}(k-1)}O_{2}, \text{ and }
a_{n-3} = 2^{k} O_{3}
\end{equation}
\noindent
with $O_{j}$ odd integers. Then the recurrence \eqref{recur-a} implies
\begin{equation}
a_{4k-3} = 2^{k} \left[ O_{1}+ (4k-5)2^{1+\nu_{2}(k-1)}O_{2} -(k-1)2^{2}O_{3} \right].
\end{equation}
\noindent
Thus $\nu_{2}(a_{4k-3}) = k$. 

\smallskip

\noindent
\textbf{Case 2}: $n = 4k$. Then 
\begin{equation}
a_{n-1} = 2^{\nu_{2}(k)+k+2}O_{1}, \, a_{n-2} = 2^{k+1}O_{2}, \text{ and }
a_{n-3} = 2^{k}O_{3}
\end{equation}
\noindent
and \eqref{recur-a} implies
\begin{equation}
a_{4k} = 2^{k} \left[ 2^{\nu_{2}(k)+3}O_{1} + 2^{2}(2k-1)O_{2} - (4k-1)O_{3} \right]
\end{equation}
\noindent
and $\nu_{2}(a_{4k}) = k$ follows form here. 

\smallskip

\noindent
\textbf{Case 3}: $n = 4k-2$. Then, as in the proof of Theorem \ref{valuation-2}, the 
recurrence needs to be iterated to produce 
\begin{equation}
a_{4k-2} = 4ka_{4k-4} + (4k-7)a_{4k-5} - 8(k-1)a_{4k-6}.
\end{equation}
\noindent
The induction hypothesis gives 
\begin{equation}
a_{4k-2} = 2^{k+1} \left[ kO_{1} + (4k-7)2^{\nu_{2}(k-1)}O_{2} - 4(k-1)O_{3} \right].
\end{equation}
\noindent
For $k$ odd, the first term in the square bracket is odd and the other two are even. For 
$k$ even, the second term is odd and the other two are even. In either case, 
$\nu_{2}(a_{4k-2}) = k+1$. 

\smallskip

\noindent
\textbf{Case 4}: $n = 4k-1$. The statement to be proved is
\begin{equation}
\nu_{2}(a_{n}) = \nu_{2}(4k) + k.
\end{equation}
\no
Observe that 
\begin{eqnarray}
a_{4k-1} & = & \sum_{j=0}^{2k-1} \frac{(2j)!}{j! 2^{j}} \binom{4k}{2j+1} \label{sum-1} \\ 
 & = & 4k \sum_{j=0}^{2k-1}  \frac{(2j)!}{j! 2^{j}}  \frac{1}{2j+1} \binom{4k-1}{2j}. \nn
 \end{eqnarray}
\no
Therefore, it suffices to show that $\nu_{2}(b_{k}) = k$ where 
\begin{equation}
b_{k} = \sum_{j=0}^{2k-1}  \frac{(2j)!}{j! 2^{j}}  \frac{1}{2j+1} \binom{4k-1}{2j}.
\label{bk-def}
\end{equation}
It should be  noted that not all summands in $b_{k}$ are integers. 

\smallskip

The proof of this last step is based on the valuations of the sum 
\begin{equation}
F(\alpha, \beta, k) = \sum_{j=0}^{2k-1} (2j+ \alpha)^{\beta} 
 \frac{(2j)!}{j! 2^{j}}   \binom{4k-1}{2j}. 
 \end{equation}
 \no
 Observe that 
 \begin{equation}
 F(\alpha,0,k) = \sum_{j=0}^{2k-1}  \frac{(2j)!}{j! 2^{j}}   \binom{4k-1}{2j} = I_{1}(4k-1).
 \end{equation}
 
 The next lemma relates  $F(\alpha,1,k)$ with  the involution numbers. 
 
 \begin{lemma}
 Let $\alpha, \, k \in \mathbb{N}$. Then 
 \begin{equation}
 F(\alpha,1,k) = \alpha I_{4k-1} + 2(4k-1)(2k-1)I_{4k-3}.
 \label{casea=1}
 \end{equation}
 \end{lemma}
 \begin{proof}
 Simply observe that 
 \begin{equation*}
 F(\alpha,1,k)  =  \alpha F(\alpha,0,k) +  
 2 \sum_{j=1}^{2k-1} j \cdot \frac{(2j)!}{j!2^{j}} \binom{4k-1}{2j}
 \end{equation*}
 \no
 and then check that the last sum is $2(4k-1)(2k-1)I_{4k-3}$.
 \end{proof}
 
The $2$-adic valuation of $F(\alpha,\beta,k)$ is computed next when $\alpha, \, \beta \in \mathbb{N}$ 
and $\alpha$ is odd. 

 \begin{theorem}
 \label{propo-val1}
 Let $\alpha, \, \beta \in \mathbb{N}$ with $\alpha$ odd. Then 
 \begin{equation}
 \label{val-alpha}
 \nu_{2}(F(\alpha, \beta,k)) = \begin{cases}
 k+1 & \quad \text{ if } \beta \text{ is even}, \\
  k & \quad \text{ if } \beta \text{ is odd}.
  \end{cases}
  \end{equation}
  \end{theorem}
  \begin{proof}
 The case $\beta=0$ is Theorem \ref{valuation-2}. The case $\beta=1$ is obtained 
 from the identity \eqref{casea=1} and Theorem \ref{valuation-2}. The rest of the proof is divided 
 according to the parity of $\beta$. 
 
 \smallskip
 
 \no
 \textbf{Case 1}: $\beta>1$ odd. Expand $(2j+\alpha)^{\beta}$ by the binomial theorem to 
 obtain 
 \begin{equation}
 F(\alpha, \beta, k) = \sum_{\ell=0}^{\beta} \binom{\beta}{\ell} \alpha^{\beta- \ell} 
 \sum_{j=0}^{2k-1} \binom{4k-1}{2j} \frac{(2j)!}{j!2^{j}} 2^{\ell} j^{\ell}.
 \end{equation}  
 \no
 The term corresponding to $\ell=0$ is 
 \begin{equation}
 t_{\ell=0} :=\alpha^{\beta} \sum_{j=0}^{2k-1} \binom{4k-1}{2j} \frac{(2j)!}{j!2^{j}}  = \alpha^{\beta} I_{1}(4k-1).
 \end{equation}
 \no
 Theorem \ref{valuation-2} gives its $2$-adic valuation as 
 \begin{equation}
 \nu_{2}(t_{\ell=0}) = \nu_{2}(I_{4k-1}) = k+1.
 \end{equation}
 \no
 The term for $\ell=1$ is 
 \begin{equation}
 t_{\ell=1} := \beta \alpha^{\beta-1} \sum_{j=0}^{2k-1} \binom{4k-1}{2j} \frac{(2j)!}{j!2^{j}} \cdot 2j = 
 2(4k-1)(2k-1)I_{1}(4k-3)
 \end{equation}
 \no
 and its $2$-adic valuation is 
 \begin{equation}
 \nu_{2}(t_{\ell=1}) = \nu_{2}(I_{4k-3}) = k.
 \end{equation}
 
 \smallskip 
 
 For  the remaining terms in the sum  $F(\alpha,\beta,k)$ use the identity 
 \begin{equation}
 j^{\ell} = \sum_{r=1}^{\ell} c_{r} \frac{j!}{(j-r)!} 
 \end{equation}
 \no
 where $c_{r} \in \mathbb{Z}$ (these are the Stirling numbers, but only their 
 integrality matters here). This leads to the expression 
 \begin{equation}
 \sum_{\ell=2}^{\beta} \binom{\beta}{\ell} \alpha^{\beta- \ell} 
 \sum_{r=1}^{\ell} c_{r} \sum_{j=0}^{2k-1} \binom{4k-1}{2j} 
 \frac{(2j)!}{(j-r)! 2^{j - \ell}}.
 \end{equation}
 \no
 The theorem now follows from the fact that the internal sum has 
 $2$-adic valuation at least $k+1$. This implies that $\ell=1$ controls the valuation. 
 In order to verify this statement, observe that 
 \begin{eqnarray*}
 \sum_{j=0}^{2k-1} \binom{4k-1}{2j} 
 \frac{(2j)!}{(j-r)! 2^{j - \ell}} & = & 2^{\ell-r} \sum_{j=0}^{2k-1} \binom{4k-1}{2j} 
 \frac{(2j)!}{(j-r)! 2^{r-j }}  \\
 & = & 2^{\ell-r} \frac{(4k-1)!}{(4k-2r-1)!} \sum_{m=0}^{2k-r} \binom{4k-2r-1}{2m} \frac{(2m)!}{m!2^{m}} \\
 & = &2^{\ell-r}  \frac{(4k-1)!}{(4k-2r-1)!} I_{4k-2r-1}.
 \end{eqnarray*}
 \no
 A direct application of Theorem \ref{valuation-2} shows that 
 \begin{equation*}
 \nu_{2} \left( \frac{2^{\ell-r}(4k-1)!}{(4k-2r-1)!} I_{4k-2r-1} \right) \geq \ell - r + \left(  r + \lf \frac{r}{2} \rf \right) + 
 \left( k + \lf \frac{r}{2} \rf -1 \right)  \geq \ell + k-1.
 \end{equation*}
 \no 
The statement about the valuation of the internal sums is now immediate since $\ell \geq 2$. 

\medskip

\no
\textbf{Case 2}: $\beta$  even. As in the case $\beta$ odd, the  valuations of the internal sums are bounded from below
by $\ell + k -1$. In particular, the lower bound is at least $k+2$ if $\ell \geq 3$. This leads to the decomposition
\begin{equation}
F(\alpha,\beta,k) = X_{1}(\alpha,\beta,k) + X_{2}(\alpha,\beta,k)+X_{3}(\alpha,\beta,k)
\end{equation}
\no
where
\begin{eqnarray}
X_{1}(\alpha, \beta,k) & = & \alpha^{\beta} \sum_{j=0}^{2k-1} \binom{4k-1}{2j} \frac{ (2j)!}{j!2^{j}}  \label{x1-def} \\
X_{2}(\alpha, \beta,k) & = & \sum_{\ell=1}^{2} \binom{\beta}{\ell} \alpha^{\beta - \ell} 
\sum_{j=0}^{2k-1} \binom{4k-1}{2j} \frac{ (2j)!}{j!2^{j}} (2j)^{\ell}  \nn \\
X_{3}(\alpha, \beta,k) & = & \sum_{\ell=3}^{\beta} \binom{\beta}{\ell} \alpha^{\beta - \ell} 
\sum_{j=0}^{2k-1} \binom{4k-1}{2j} \frac{ (2j)!}{j!2^{j}} (2j)^{\ell}.  \nn
\end{eqnarray}
\no
Then $\nu_{2}(X_{3}(\alpha,\beta,k)) \geq k+2$. It is now shown that $\nu_{2}(X_{1}(\alpha,\beta,k)) = k+1$ and  
$\nu_{2}(X_{2}(\alpha,\beta,k)) \geq  k+3$. This proves the formula for the valuation of $F(\alpha, \beta, k)$ when 
$\beta$ is even. 

\medskip

To prove the statement about the valuation of $X_{1}$ use the identity $X_{1}(\alpha,\beta,k) = \alpha^{\beta} I_{1}(4k-1)$  and Theorem \ref{valuation-2}. The proof of the corresponding formula for  $X_{2}$ starts with the expression
\begin{eqnarray}
X_{2}(\alpha,\beta,k) & = & \beta \alpha^{\beta-1} \sum_{j=0}^{2k-1} \binom{4k-1}{2j} \frac{ (2j)!}{j!2^{j}} (2j) \label{x2-one}  \\
& + & \binom{\beta}{2} \alpha^{\beta-2} \sum_{j=0}^{2k-1} \binom{4k-1}{2j} \frac{ (2j)!}{j!2^{j}} (4j^{2}) \nn 
\end{eqnarray}
\no
and then use $4j^{2} = 4j(j-1) + 4j$ and the identities
\begin{eqnarray*}
 \sum_{j=0}^{2k-1} \binom{4k-1}{2j} \frac{ (2j)!}{j!2^{j}} (2j) & = & (4k-1)(4k-2) I_{4k-3} \\
 \sum_{j=0}^{2k-1} \binom{4k-1}{2j} \frac{ (2j)!}{j!2^{j}} (4j^{2})  & = & (4k-1)(4k-2)(4k-3)(4k-4)I_{4k-5} \nn
 \end{eqnarray*}
\no
to arrive at 
\begin{eqnarray*}
X_{2}(\alpha,\beta,k)   & = &  2 \beta \alpha^{\beta-2} (\alpha + \beta-1)(2k-1)(4k-1)I_{4k-3}  \\
& + & 8 \binom{\beta}{2} \alpha^{\beta-2} (4k-1)(2k-1)(4k-3)(k-1)I_{4k-5}.
\end{eqnarray*}
\no
Then Theorem \ref{valuation-2} implies 
\begin{equation}
\nu_{2}(\beta) + \nu_{2}(\alpha+\beta-1) + 1 + \nu_{2}(I_{4k-3}) = 
\nu_{2}(\beta) + \nu_{2}(\alpha+\beta-1) + k \geq k+2,
\end{equation}
\no
and the valuation of the second term is 
\begin{equation}
\nu_{2}(\beta) -1 + 3  + \nu_{2}(I_{4k-5}) = \nu_{2}(\beta) + 2  + k \geq k+3.
\end{equation}
\no
The statement about  $\nu_{2}(X_{2})$ is established . The formula for $\nu_{2}(F(\alpha,\beta,k))$, when $\beta$ is even, follows from these results.
 \end{proof}

\medskip

The remainder of the  proof of Theorem \ref{valuation-an} has been reduced to verifying that
 $\nu_{2}(b_{k}) = k$, where $b_{k}=F(1,-1,k)$ is defined  in \eqref{bk-def}.
 
 \medskip

For $m \in \mathbb{N}$ and $a$ odd, Euler's theorem yields 
\begin{equation}
a^{-1} \equiv a^{\varphi(2^{m}) -1} = a^{2^{m-1}-1} \bmod 2^{m}.
\end{equation}
\no
Therefore 
\begin{eqnarray*}
F(1,-1,k)  & \equiv &  \sum_{j=0}^{2k-1} (2j+1)^{2^{m-1}-1} \frac{(2j)!}{j! 2^{j}} \binom{4k-1}{2j}  \\
& = & F(1,2^{m-1}-1,k) \bmod 2^{m}.
\end{eqnarray*}
\no
Since $2^{m-1}-1$ is odd, Proposition \ref{propo-val1} gives 
\begin{equation}
\nu_{2}(F(1,2^{m-1}-1,k)) = k.
\end{equation}
\no
Now choose $m=k$ to compute 
\begin{equation}
F(1,-1,k) \equiv F(1,2^{k-1}-1,k) \equiv 0 \bmod 2^{k}
\end{equation}
\no
and then choose $m=k+1$ to obtain
\begin{equation}
F(1,-1,k) \equiv F(1,2^{k}-1,k) \not \equiv 0 \bmod 2^{k+1}.
\end{equation}
\no
It follows that $\nu_{2}(F(1,-1,k)) = k$, as desired.  The proof 
of Theorem \ref{valuation-an} is now complete.
\end{proof}

\begin{note}
For $p$ odd, the $p$-adic valuation of $a_{n}$ also exhibits some interesting 
patterns which will be investigated  in the future.  For instance, when  $p=3$, it is noted that 
\begin{equation}
\nu_{3}(a_{n}) = 0 \text{  if }n \not \equiv 8 \bmod 9
\end{equation}
\no
and 
\begin{equation}
\nu_{3}(a_{9n+8}) = \begin{cases}
0 & \quad \text{if} \quad n \equiv 0 \bmod 3, \\
0 & \quad \text{if} \quad n \equiv 1 \bmod 3, \\
\nu_{3}(n+1) & \quad \text{if} \quad n \equiv 2 \bmod 3.
\end{cases}
\end{equation}
\no
Similar formulas may be tested out experimentally  for other primes.
\end{note}

\section{Permutations of restricted length}
\label{sec-permu}
\setcounter{equation}{0}

Let $n \in \mathbb{N}$ and $0 \leq \ell \leq n$.  This section considers the 
set 
\begin{equation*}
C_{n, \ell} = \{ \pi \in \mathfrak{S}_{n} \Big{|} \text{ every cycle in } \pi \text{ is of length at most } \ell \}
\end{equation*}
\no
and its  cardinality $d_{n,\ell} = \# C_{n,\ell}$.

\begin{proposition}
\label{gen-dn}
The numbers $d_{n,\ell}$ satisfy the recurrence 
\begin{equation*}
d_{n+1,\ell} = d_{n,\ell} + \frac{n!}{(n-1)!} d_{n-1,\ell} + \frac{n!}{(n-2)!}d_{n-2,\ell} +
\cdots + \frac{n!}{(n-\ell+1)!} d_{n+1 - \ell,  \ell}.
\end{equation*}
\no
Equivalently
\begin{equation}
\sum_{n=0}^{\infty} d_{n,\ell} \frac{x^{n}}{n!} = { \rm{exp}}\left( x + \frac{x^{2}}{2} + \frac{x^{3}}{3}+ 
\cdots + \frac{x^{\ell}}{\ell} \right).
\end{equation}
\end{proposition}
\begin{proof}
For $1 \leq j \leq \ell$, the number $1$ is in exactly $n!/(n-j+1)!$ cycles of length $j$. Now count the 
elements in $C_{n+1,\ell}$ according to the length of the cycle containing the number $1$. 
\end{proof}

A multivariate generalization of Proposition \ref{gen-dn} is given in terms of Toeplitz matrices. 

\begin{definition}
For $n, \, \ell \in \mathbb{N}$ and indeterminates $Y_{1}, \, Y_{2}, \cdots, Y_{\ell}$, the 
square matrix $M_{n,\ell}(\mathbf{Y})$, of size $n+1$,  has entries
\begin{equation}
M_{n, \ell}(k,j) = 
\begin{cases}
\imath^{j-k} Y_{j-k+1} & \quad \text{ if } 0 \leq j -k \leq \ell - 1, \\
 \imath j& \quad \text{ if } k = j+1, \\
0, & \quad \text{ otherwise}.
\end{cases}
\end{equation}
\end{definition}

\begin{example}
Let $n=5$ and $\ell=4$, then
\begin{equation}
M_{5,4}(\mathbf{Y}) = 
\begin{pmatrix}
Y_{1} & \imath Y_{2} & \imath^{2}Y_{3} & \imath^{3} Y_{4} & 0 \\
\imath & Y_{1} & \imath Y_{2} & \imath^{2} Y_{3} & \imath^{3} Y_{4} \\
0 & 2 \imath & Y_{1} & \imath Y_{2} & \imath^{2} Y_{3} \\
0 & 0 & 3 \imath & Y_{1} & \imath Y_{2} \\
0 & 0 & 0 & 4 \imath & Y_{1} 
\end{pmatrix}.
\end{equation}
\end{example}

\medskip

A generalization of Proposition \ref{gen-dn} is stated next. 

\begin{Thm}
The exponential generating function for the determinants of 
$M_{n,\ell}(\mathbf{Y})$ is 
\begin{equation}
\sum_{n=0}^{\infty} \det( M_{n, \ell}(\mathbf{Y}))  \frac{x^{n}}{n!} = 
\text{exp} \left( Y_{1} x + Y_{2} \frac{x^{2}}{2} + \cdots + 
Y_{\ell} \frac{x^{\ell}}{\ell} \right).
\label{gf-identity}
\end{equation}
\end{Thm}
\begin{proof}
Fix $\ell$ and let $g_{n,\ell}(\mathbf{Y})= \text{det} (M_{n, \ell}(\mathbf{Y}))$. Use 
Laplace expansion of $g_{n,\ell}$ along the last row to obtain
\begin{multline}
g_{n,\ell}(\mathbf{Y})  = g_{n-1,\ell}(\mathbf{Y})Y_{1} + (n-1)g_{n-2,\ell}(\mathbf{Y})Y_{2} + (n-1)(n-2)g_{n-3,\ell}(\mathbf{Y})Y_{3} + \cdots + \\
+ (n-1)(n-2) \cdots (n-\ell+1)g_{n- \ell,\ell}(\mathbf{Y})Y_{\ell}.  \nonumber
\end{multline}
\no
Now set $\begin{displaystyle} F_{\ell}(x;\mathbf{Y}) = \sum_{n=0}^{\infty} g_{n,\ell} (\mathbf{Y})\frac{x^{n}}{n!}
\end{displaystyle}$. The recurrence shows that  $F_{\ell}(x;\mathbf{Y})$ matches the
 right-hand side of \eqref{gf-identity}.
\end{proof}

\smallskip

Comparing coefficients in the expansion \eqref{gf-identity} gives a statistic on the 
set $C_{n, \ell}$. 

\medskip

\begin{Thm}
Let $n \in \mathbb{N}$ and $0 \leq \ell \leq n$. Recall
$\alpha_{t}(\pi) = $ number of  $t$-cycles  in  $ \pi \in \mathfrak{S}_{n}$. Then 
\begin{equation}
\text{det }M_{n, \ell}(Y_{1}, \cdots, Y_{\ell}) = 
\sum_{\pi \in C_{n, \ell}} Y_{1}^{\alpha_{1}(\pi)} \cdots Y_{2}^{\alpha_{\ell}(\pi)}.
\end{equation}
\end{Thm}

\begin{example}
Let  $n=5$ and $\ell=4$. Then the cycle-index polynomial is computed as
\begin{eqnarray*}
\det M_{5,4}(\mathbf{Y}) & = & 
Y_{1}^{5} + 10Y_{1}^{3}Y_{2} + 20 Y_{1}^{2}Y_{3} + 15Y_{1}Y_{2}^{2} +
30 Y_{1}Y_{4} + 20Y_{2}Y_{3} \\
& = & \sum_{\pi \in C_{5,4}} Y_{1}^{\alpha_{1}(\pi)} Y_{2}^{\alpha_{2}(\pi)} 
Y_{3}^{\alpha_{3}(\pi)} Y_{4}^{\alpha_{4}(\pi)} 
\end{eqnarray*}
\no
and it encodes the statistic on the set $C_{5,4}$. For instance,
$20$ permutations in $\mathfrak{S}_{5}$ are a product of a $2$-cycle and a $3$-cycle. 
Also, there are $\# C_{5,4} = \det M_{5,4}(\mathbf{1}) = 96$ permutations 
formed by cycles of length $4$ or less. This is $5!=120$ minus the $24$ 
cycles of length $5$.
\end{example}

\smallskip

\begin{note}
The special case $Y_{j}=1$ (for all $j$) produces 
\begin{equation}
d_{n,\ell} = \text{det}(M_{n,\ell})(\mathbf{1}).
\end{equation}.
\end{note}

\section{Asymptotics}
\label{sec-asymptotics}
\setcounter{equation}{0}

This section considers the asymptotic behavior, as $n \to \infty$ with $\ell$ fixed, of the 
numbers $d_{n, \ell}  =  \# C_{n,\ell}$, counting the number of permutations 
in $\mathfrak{S}_{n}$ with every cycle of length at most $\ell$. Their 
exponential  generating function is 
\begin{equation}
f_{\ell}(z) = \sum_{n=0}^{\infty} d_{n,\ell} \frac{z^{n}}{n!} = \exp\left(z + \frac{z^{2}}{2} + \cdots + \frac{z^{\ell}}{\ell} \right).
\end{equation}
\no
Several authors provide asymptotic expansions for $f_{\ell}(z)$ (see Moser-Wyman \cite{moser-1955a} and 
Knuth \cite{knuth-1973a} when $\ell=2$; Wimp-Zeilberger \cite{wimp-1985a} for any $\ell$ using methods from
Birkhoff-Trjitzinsky). In this section, the general case is revisited using the \textit{saddle-point} technique in 
order to generate a first-order estimate.

Cauchy's integral formula gives 
\begin{equation}
d_{n,\ell} = \frac{n!}{2 \pi \imath} \oint_{C} \frac{f_{\ell}(z)}{z^{n+1}} dz
\end{equation}
\no
where $C$ is a simple closed curve around the origin. In the analysis presented here $C$ is a circle 
of radius $r$. Therefore 
\begin{eqnarray}
d_{n,\ell} & = & \frac{n!}{2 \pi \imath} \int_{|z|=r} f_{\ell}(z) \exp(-n \log z) \frac{dz}{z} \\
& = & \frac{n!}{2 \pi \imath} \int_{|z|=r} \text{exp}\left( z + \frac{z^{2}}{2} + \cdots + \frac{z^{\ell}}{\ell} -n \log z \right)\frac{dz}{z}
\nonumber
\end{eqnarray}
\no
The saddle-point method \cite{millerp-2006a} gives 
\begin{equation}
d_{n,\ell} \sim \frac{n!}{\sqrt{2 \pi \ell n}} \exp\left(r + r^{2}/2 + \cdots + r^{\ell}/\ell - n \log r \right)
\label{saddle-3}
\end{equation}
\no
where the saddle point $r_{+} \in \mathbb{R}^{+}$ is defined by the equation 
\begin{equation}
\frac{d}{dr} \left( r + r^2/2 + \cdots + r^{\ell}/\ell- n \log r \right) = 1 + r + r^{2} + \cdots + r^{\ell-1} - \frac{n}{r} = 0.
\label{saddle-1}
\end{equation}
\no
This is equivalent to $r + r^{2} + \cdots + r^{\ell} = n$. To 
 obtain information about the saddle point $r_{+}$, it is convenient to rewrite \eqref{saddle-1} in the form
\begin{equation}
r \left( \frac{1 - r^{-\ell}}{1 - r^{-1}} \right)^{1/\ell} - \eta = 0, \, \quad \text{ for } \eta \in \mathbb{C}.
\label{saddle-2}
\end{equation}
\no
This defines $r = r(\eta)$.  For $\eta = n^{1/\ell}$, this becomes $r(n^{1/\ell}) = r_{+}$. 

\smallskip

The limiting value $r/\eta \to 1$, as $r \to \infty$, suggests the asymptotic expansion 
\begin{equation}
r = r(\eta) = \eta + \alpha_{0} + \frac{\alpha_{-1}}{\eta} + \frac{\alpha_{-2}}{\eta^{2}} + \cdots, \text{ as } \eta \to \infty.
\label{saddle-4}
\end{equation}
\no
The exponent in \eqref{saddle-3} is now written as $\Phi(\eta) - \eta^{\ell} \log \eta$, with 
\begin{equation}
\Phi(\eta) = r(\eta)+ \frac{1}{2}r(\eta)^{2} + \cdots + \frac{1}{\ell} r(\eta)^{\ell} - \eta^{\ell} \log \frac{r(\eta)}{\eta}.
\end{equation}
\no
The expansion \eqref{saddle-4} leads to
\begin{equation}
\Phi(\eta) = \beta_{\ell} \eta^{\ell} + \cdots + \beta_{1} \eta + \beta_{0} + \frac{\beta_{-1}}{\eta} + \cdots
\end{equation}
\no
and the relevant contributions to the behavior of \eqref{saddle-3} come from the positive powers in this expansion. To compute these 
contributions observe that, for $k >0$,
\begin{equation}
\beta_{k} = \frac{1}{2 \pi \imath}  \int_{C} \frac{\Phi(\eta)}{\eta^{k+1}} d \eta = \frac{1}{2 \pi \imath } \int_{C} \frac{\Phi'(\eta)}{k \eta^{k}} d \eta,
\end{equation}
\no
where the second expression is obtained by using Cauchy's integral formula on the expansion of $\Phi'(\eta)$.  The contour $C$ is made to pass through the point $r_{+}$. Therefore, using $\Phi'(\eta) = \eta^{\ell-1} - \ell \eta^{\ell - 1} \log \frac{r}{\eta}$ 
it follows that, for $0 < k < \ell$, 
\begin{eqnarray}
\beta_{k}  & = & \frac{1}{2 \pi \imath } \int_{C} \frac{1}{k \eta^{k - \ell + 1}} d \eta - \frac{1}{2 \pi i } 
\int_{C} \frac{\ell \log \frac{r}{\eta}}{k \eta^{k - \ell + 1} }d \eta \\
& = & - \frac{1}{2 \pi \imath  } 
\int_{C} \frac{\ell \log \frac{r}{\eta}}{k \eta^{k - \ell + 1}}d \eta \nonumber
\end{eqnarray}
\no
since the first integral vanishes. This is now written as 
\begin{eqnarray}
\beta_{k} & = &  - \frac{1}{2 \pi \imath} \frac{\ell}{k(\ell-k)}  \int_{C} \log \frac{r}{\eta} \frac{d}{d \eta} \eta^{\ell - k} \, d \eta  \label{int-91a} \\
& = & \frac{1}{2 \pi \imath }  \frac{\ell}{k(\ell-k)} \int_{C} \left( \frac{r'}{r} - \frac{1}{\eta} \right) \eta^{\ell - k} d \eta \nonumber \\
& = & \frac{1}{2 \pi \imath}  \frac{\ell}{k(\ell-k)}  \int_{C} \frac{r'}{r} \eta^{\ell - k} d \eta \nonumber
\end{eqnarray}
\no
since the integral of $\eta^{\ell - k - 1}$ vanishes. Now make the change of variables $r \mapsto r(\eta)$ to obtain 
\begin{equation}
\beta_{k}  =  \frac{1}{2 \pi \imath } \int_{C_{1}} \frac{\ell}{k(\ell-k)} \left( \frac{1 - r^{- \ell} }{1 - r^{-1}} \right)^{\tfrac{\ell -k}{k}} 
r^{\ell - k - 1} dr.
\end{equation}
A residue calculation then gives 
\begin{equation}
\beta_{k} = \frac{1}{k(\ell - k)!} \left( \frac{\ell-k}{\ell} + 1 \right) \cdots \left( \frac{\ell - k}{\ell} + \ell - 1 \right).
\end{equation}

\smallskip

An easier calculation, left to the reader, gives 
$\begin{displaystyle}\beta_{0} = - \frac{1}{\ell} \sum_{j=2}^{\ell} \frac{1}{j} \end{displaystyle}$ and 
$\begin{displaystyle}\beta_{\ell} = \frac{1}{\ell} \end{displaystyle}$.  Combining the above and invoking Stirling's formula
 $n! \sim \sqrt{2 \pi n } n^{n} e^{-n}$ produces  the following result. 

\begin{theorem}
Let $\ell \in \mathbb{N}$ be fixed. Then
\begin{equation*}
d_{n, \ell} = \frac{1}{\sqrt{\ell}} n^{n(1 - 1/\ell)}  
\exp \left( - \frac{1}{\ell} \sum_{j=2}^{\ell} \frac{1}{j} + \frac{n}{\ell}  + \sum_{k=1}^{\ell - 1} 
\frac{ \left( \tfrac{k}{\ell} + 1 \right) \cdots \left( \tfrac{k}{\ell} + \ell - 1 \right)}{k! (\ell -k )} n^{\tfrac{\ell - k}{\ell}} \right).
\end{equation*}
\end{theorem}

\bigskip

\no
\textbf{Acknowledgments}. The second  author acknowledges the
partial support of NSF-DMS 1112656.


\begin{thebibliography}{1}

\bibitem{knuth-1973a}
D.~E. Knuth.
\newblock {\em Art of {C}omputer {P}rogramming. {S}orting and {S}earching},
  volume~3.
\newblock Addison-Wesley, Reading, Mass., 1st. edition, 1997.

\bibitem{millerp-2006a}
P.~D. Miller.
\newblock {\em Applied {A}symptotic {A}nalysis}, volume~75 of {\em Graduate
  {S}tudies in {M}athematics}.
\newblock American {M}athematical {S}ociety, 2006.

\bibitem{moser-1955a}
I.~Moser and M.~Wyman.
\newblock On the solutions of $x^{d}=1$.
\newblock {\em Canad. J. Math.}, 7:159--168, 1955.

\bibitem{nemesi-1997a}
I.~Nemes, M.~Petkovsek, H.~Wilf, and D.~Zeilberger.
\newblock How to do {MONTHLY} problems with your computer.
\newblock {\em Amer. {M}ath. {M}onthly}, 104:505--519, 1997.

\bibitem{petkovsek-1996a}
M.~Petkovsek, H.~Wilf, and D.~Zeilberger.
\newblock {\em A=B}.
\newblock A.~K.~Peters, 1st. edition, 1996.

\bibitem{wimp-1985a}
J.~Wimp and D.~Zeilberger.
\newblock Resurrecting the asymptotics of linear recurrences.
\newblock {\em J. Math. Anal. Appl.}, 111:162--176, 1985.

\end{thebibliography}
\end{document}